\documentclass[11pt,reqno,oneside]{amsart}

\usepackage[latin1]{inputenc}
\usepackage{amssymb,amsmath,amstext,amsthm,amsfonts}
\usepackage{amsfonts}
\usepackage{amsmath}
\usepackage{amssymb}
\usepackage{amsthm}
\usepackage{pstricks,pstricks-add,pst-math,pst-xkey}

\newcommand{\bdm}{\begin{displaymath}}
\newcommand{\edm}{\end{displaymath}}

\newcommand{\rd}{\mathbb{R}^d}
\newcommand{\xa}{X_{\bar{a}}}
\newcommand{\freq}{\mathrm{freq}}

\theoremstyle{definition}
\newtheorem{thm}{Theorem}

\newtheorem{lem}[thm]{Lemma}

\newtheorem{prop}[thm]{Proposition}
\newtheorem{rem}{Remark}
\newtheorem{eg}{Example}

\newtheorem{cor}[thm]{Corollary}

\title[Statistical Stability for Multi-Substitution Tiling Spaces]{Statistical Stability  for \\ Multi-Substitution Tiling Spaces}

\author[R. Pacheco]{Rui Pacheco}
\address{Universidade da Beira Interior\\
Rua Marquês d'Ávila e Bolama, 6200-001 Covilhã, Portugal}
\email{rpacheco@ubi.pt} \urladdr{http://www.mat.ubi.pt/$\sim$rpacheco}

\author[H. Vilarinho]{Helder Vilarinho}
\address{Universidade da Beira Interior\\
Rua Marquês d'Ávila e Bolama, 6200-001 Covilhã, Portugal}
\email{helder@ubi.pt} \urladdr{http://www.mat.ubi.pt/$\sim$helder}

\date{\today}

\thanks{The authors were partially supported by the Portuguese Government through FCT, under the project PEst-OE/MAT/UI0212/2011 (CMUBI)}

\keywords{multi-substitutions, tiling spaces, dynamical systems, invariant measures, statistical stability}
\begin{document}
\begin{abstract}
Given a finite set $\{S_1\dots,S_k  \}$ of substitution maps acting on a certain finite number (up to translations) of tiles in $\rd$, we consider the multi-substitution tiling space associated to each  sequence $\bar a\in \{1,\ldots,k\}^{\mathbb{N}}$. The action by translations on such spaces gives rise to uniquely ergodic dynamical systems.
In this paper we investigate the rate of convergence for  ergodic limits  of patches frequencies  and prove that these limits  vary continuously with  $\bar a$.
\end{abstract}

\subjclass[2010]{37A15, 37A25, 52C22}
\maketitle

\section{Introduction}

Roughly speaking, a \emph{tiling} of $\mathbb{R}^d$ is an arrangement of tiles that covers $\mathbb{R}^d$ without overlapping. An important class of tilings is that of \emph{self-similar tilings}.
In order to construct a self-similar tiling $x$, one starts with a finite number (up to translation) of tiles  and a \emph{substitution map} $S$ that determines how to inflate  and subdivide these tiles into  certain configurations of the same tiles. Many examples can be found in \cite{F,GS}. The \emph{substitution tiling space} $X_S$  is the closure of all the translations of $x$ in an appropriate metric, with respect to which $X_S$ is compact and the group $\rd$ acts continuously on $X_S$ by translations, defining a \emph{substitution  dynamical system}. The ergodic and spectral properties of such dynamical systems were studied in detail by Solomyak \cite{Sol}.

A substitution tiling space $X_S$ is then associated to an hierarchy. The zero level  of this hierarchy is constituted  by the initial set of tiles and  the level $i>0$ is constituted by the patches of tiles obtained from those of level $i-1$ by applying the substitution map $S$. Recently, Frank and Sadun \cite{FS} have introduced a framework to handling with  general hierarchical (\emph{fusion}) tiling spaces, where the procedure for obtaining patches of level $i$ from those of level $i-1$ is not necessarily an ``inflate-subdivide" procedure and can depend on $i$.  However,
many of the ergodic and spectral properties available for substitution dynamical systems  are hard to achieve in such generality.

In the present paper we deal with  \emph{multi-substitution tiling spaces}, also referred to in the literature as \emph{mixed substitution} tiling spaces \cite{GM} or \emph{$S$-adic systems \cite{Du,Fe}}. They form a particular class of hierarchical  tiling spaces which includes the substitution tiling spaces. A multi-substitution tiling space is determined by a finite number  (up to translation) of tiles, a finite set $\mathcal{S}=\{S_1\dots,S_k  \}$ of substitutions maps acting on these tiles and a sequence $\bar a=(a_1,a_2,\ldots)$ in $\Sigma:=\{1,\ldots,k\}^{\mathbb{N}}$. In the corresponding hierarchy, the patches of the level $i$ are obtained from those of level $i-1$ by applying the substitution map $S_{a_{i}}$. The continuous action of $\rd$ by translations on a multi-substitution tiling space $X_{\bar {{a}}}(\mathcal{S})$ defines a uniquely ergodic  dynamical system. The unique measure $\mu_{\bar a,\mathcal{S}}$ is closely related with the patch  frequencies  in tilings of $X_{\bar{{a}}}(\mathcal{S})$. In this paper we prove that, in the usual topology of $\Sigma$, the ergodic limits of patch frequencies vary continuously with  $\bar a$ (theorem \ref{ss}).  Moreover, we  prove that
the convergence of patch frequencies to their ergodic limits is locally uniform in some open subset of $\Sigma$ (theorem \ref{cu}).

\section{Tilings and Substitutions}
We start by recalling some standard definitions and results concerning substitution tiling spaces. For details, motivation and examples see \cite{F,GS,Rob,Sol}. We introduce also the concept of \emph{strongly recognizable} substitution. As we will see later, such substitutions provide isomorphisms between  ergodic dynamical systems associated to certain multi-substitution tiling spaces.

\vspace{.20in}

Consider $\mathbb{R}^d$ with its usual Euclidean norm $\|\cdot\|$ and write $B_r=\{\vec v\in\mathbb{R}^d:\,\|\vec v\|\leq r\}$. A set $D\subset \mathbb{R}^d$ is called a \emph{tile} if it is compact, connected and equal to the closure of its interior. A \emph{patch}  is a collection  $x=\{D_i\}_{i\in I}$ of tiles such that $D^{^\circ}_i\cap D^{^\circ}_j=\emptyset$, for all $i,j\in I$ with $i\neq j$. The \emph{support} of $x$ is defined by   $\mathrm{supp}(x):=\bigcup_{i\in I}D_i$.
 If  $\mathrm{supp}(x)=\rd$, we say that $x$ is a \emph{tiling} of $\rd$. When a patch has a single tile $D$, we identify this patch with the corresponding tile. Given a patch $x=\{D_i\}_{i\in I}$ and $\vec{t}\in\mathbb{R}^d$, $\vec{t}+x:=\{\vec{t}+D_i\}_{i\in I}$ is another patch. In particular, if $x$ is a tiling of $\rd$, $\vec{t}+x$ is another tiling of $\rd$. Hence we have an  action of $\mathbb{R}^d$ on the space of all tilings of $\rd$ by translations, which we denote by $T$.
  Two patches $x$ and $x'$ are said to be \emph{equivalent} if $x'=\vec{t}+x$ for some $\vec{t}\in\mathbb{R}^d$. We denote by $[x]$ the equivalence class of $x$.

 Let $X$ be a space of tilings of $\rd$ invariant by $T$
and  $\mathcal{P}^N(X)$ be the set of all patches $x'=\{D_i\}_{i\in I}$ such that $|I|=N$ and $x'\subset x$ for some $x\in  X$.  We denote by $\mathcal{T}^N(X)$ the set of equivalence classes with representatives in $\mathcal{P}^N(X)$. These representatives  are called $N$-\emph{protopatches} of $X$. The $1$-protopatches are more usually called \emph{prototiles}. The tiling space $X$ has \emph{finite local complexity} if $\mathcal{T}^2(X)$ is finite. Equivalently, $\mathcal{T}^N(X)$ is finite for each $N$.

 If $K\subset \mathbb{R}^d$ is compact and $x\in X$, we denote by $x[[K]]$ the set of all patches $x'\subset x$ with bounded support satisfying $K\subseteq \mathrm{supp}( x')$.
  For $x,y\in X$, we set
   \begin{align*}
  \nonumber  d_T(x,y)=\inf\Big\{\{\sqrt{2}/2\}\cup&\{0<r<\sqrt{2}/2: \,\textrm{exist $x'\in x[[B_{1/r}]]$,  $\,y'\in y[[B_{1/r}]],$}\\ &\textrm{and $\vec{t}\in \mathbb{R}^d$ with $\|\vec{t}\|\leq r$ and $\vec{t}+x'=y'$}\}\Big\}.\label{distance}
  \end{align*}

  \begin{thm}\cite{Rob,Sol}
   $(X,d_T)$ is a  complete metric space. Moreover, if $X$ has finite local complexity, then $(X,d_T)$ is compact and the action  $T$ is continuous.
  \end{thm}

   From now on we assume that $X$ is equipped with the metric $d_T$ and that $X$  has finite local complexity.

\begin{rem}
  The above equivalence relation between patches and the corresponding definition of distance could be defined with respect to rather general ``actions" of groups on patches (see \cite{PV}).  However, the metric $d_T$ is adequate to   the purposes of this paper since we shall only be concerned with the dynamics associated to the action $T$.
\end{rem}

  A \emph{(self-similar) substitution} is a map $S:\mathcal{P}^1(X)\to\mathcal{P}(X):=\bigcup_N\mathcal{P}^N(X)$ such that:
\begin{itemize}
    \item[(S$_1$)] there is $\lambda>1$ (the \emph{dilatation factor} of $S$) such that
  $\mathrm{supp}(S(P))=\lambda \mathrm{supp}(P)$ for all $P\in\mathcal{P}^1(X)$;
  \item[(S$_2$)] if $P=\vec{t}+Q$ then $S(P)=\lambda \vec{t}+S(Q)$.
  \end{itemize}

Take a finite number of (non-equivalent) prototiles $\{D_1,\ldots,D_l\}$ of $X$ such that  $$\mathcal{T}^1(X)=\{[D_1],\ldots, [D_l]\}.$$ The \emph{structure matrix} $A_S$ associated to the substitution $S$ is the $l\times l$ matrix with entries $A_{ij}$ equal to the number of tiles equivalent to $D_i$ that appear in $S(D_j)$. If $A_S^m>0$ for some $m>0$, $S$ is said to be \emph{primitive}. In the particular case $m=1$, $S$ is \emph{strongly primitive}.

  Given a patch $x=\{D_i\}_{i \in I}$ with $D_i\in \mathcal{P}^1(X)$, we define the patch $S(x):=\bigcup_{i\in I} S(D_i)$. Assume that the substitution $S$ can be extended to maps $S: \mathcal{P}(X)\to \mathcal{P}(X)$ and $S:X\to X$.  In this case, take a tile $D\in \mathcal{P}^1(X)$ and  define inductively the following sequence of patches in $\mathcal{P}(X)$: $x_1=D$, and $x_k=S(x_{k-1})$ for $k>1$. Consider the closed (hence compact) tiling space $X_S\subseteq X$, commonly known as  \emph{substitution tiling space} associated to $S$, defined by: a tiling $x\in X$ belongs to $X_S$ if, and only if, for any finite patch $x'\subset x$ there exist $k>0$ and a vector $\vec{t}\in \rd$ such that $\vec{t}+x'\subseteq x_k$. We have:

  \begin{prop}\cite{Rob,Sol} Suppose that $S$ is primitive. Then $X_S\neq \emptyset$, $S(X_S)\subseteq X_S$ and $X_S$ is independent of the initial tile $D\in\mathcal{P}^1(X)$.
   \end{prop}

The Perron-Frobenius (PF) theorem for non-negative matrices is a crucial tool for the study of  substitution tiling spaces:
\begin{thm}\cite{Rue}
  Let $A\geq 0$ be a real square matrix with $A^m>0$ for some $m>0$. Then there is a simple positive eigenvalue $\omega >0$ of $A$ with $\omega>|\omega'|$ for all other eigenvalues $\omega'$. Moreover, there exist  eigenvectors  $\vec p$ and $\vec q$   corresponding to $\omega$ for $A$ and $A^T$, respectively, such that $\vec p\cdot \vec q=1$ and $\vec p,\vec q >0$. In this case,  for any $\vec v\in \rd$
  $$\lim_n \frac{A^n\vec v}{\omega^n}=(\vec q\cdot \vec v)\vec p.$$
 \end{thm}
The eigenvalue $\omega >0$  is called the \emph{PF-eigenvalue} of $A$. The eigenvectors $\vec p >0$ and $\vec q>0$ are called the \emph{right PF-eigenvector} and \emph{left PF-eigenvector}  of $A$, respectively.

In general, if $S$ is a substitution acting on a set of prototiles $\{D_1,\ldots,D_l\}$ with Euclidian volumes $V_1,\ldots,V_l$, the vector $\vec q=(V_1,\ldots,V_l)$ is a left eigenvector of $A_S$ associated to the eigenvalue  $\lambda^d$. For primitive substitutions,  $\vec q$ and $\omega=\lambda^d$ are precisely the left PF-eigenvector  and the corresponding PF-eigenvalue of $A_S$, respectively.

A substitution $S$ is said to be \emph{recognizable} if $S:X\to S(X)$ is injective. In this case, we say that $S$ is \emph{strongly recognizable} if for any $x\in X$ and any tile $D\in \mathcal{P}^1(X)$ the following holds: if $S(D)$ is a patch of $S(x)$, then $D\in x$.

\begin{eg}
The Ammann A3 substitution (figure \ref{ammann}; see \cite{GS} for a detailed description of this substitution) is recognizable but not strongly  recognizable.
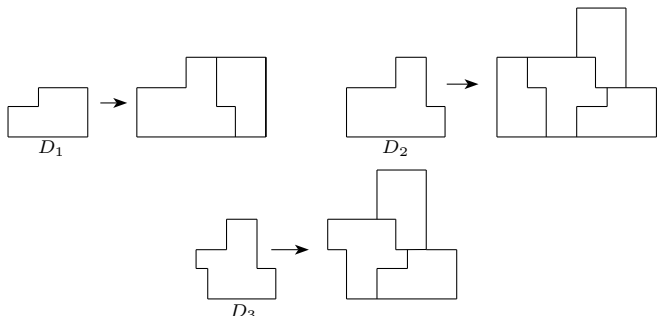
\begin{figure}[!htb]\label{ammann}
\psset{xunit=0.25cm,yunit=0.25cm,algebraic=true,dotstyle=*,dotsize=3pt 0,linewidth=0.4pt,arrowsize=3pt 2,arrowinset=0.25}
\begin{pspicture*}(-6.65,-9.5)(46.45,8)
\psline(0,0)(0,1.62)
\psline(0,1.62)(1.62,1.62)
\psline(1.62,1.62)(1.62,2.62)
\psline(1.62,2.62)(4.24,2.62)
\psline(4.24,2.62)(4.24,0)
\psline(4.24,0)(0,0)
\psline(13.71,4.24)(13.71,0)
\psline(13.71,4.24)(13.71,0)
\psline(13.71,4.24)(13.71,0)
\psline(9.47,4.24)(13.71,4.24)
\psline(9.47,2.62)(9.47,4.24)
\psline(6.85,2.62)(9.47,2.62)
\psline(6.85,0)(6.85,2.62)
\psline(18,0)(18,2.62)
\psline(18,2.62)(20.62,2.62)
\psline(20.62,2.62)(20.62,4.24)
\psline(20.62,4.24)(22.24,4.24)
\psline(22.24,4.24)(22.24,1.62)
\psline(22.24,1.62)(23.24,1.62)
\psline(23.24,1.62)(23.24,0)
\psline(23.24,0)(18,0)
\psline(26,0)(26,4.24)
\psline(26,4.24)(30.24,4.24)
\psline(30.24,4.24)(30.24,6.86)
\psline(30.24,6.86)(32.86,6.86)
\psline(32.86,6.86)(32.86,2.62)
\psline(32.86,2.62)(34.48,2.62)
\psline(34.48,2.62)(34.48,0)
\psline(34.48,0)(26,0)
\psline(27.62,4.24)(27.62,2.62)
\psline(27.62,2.62)(28.62,2.62)
\psline(28.62,2.62)(28.62,0)
\psline(30.24,4.24)(31.24,4.24)
\psline(31.24,4.24)(31.24,2.62)
\psline(31.24,2.62)(32.86,2.62)
\psline(30.24,0)(30.24,1.62)
\psline(30.24,1.62)(31.86,1.62)
\psline(31.86,1.62)(31.86,2.62)
\psline(11.09,4.24)(11.09,1.62)
\psline(11.09,1.62)(12.09,1.62)
\psline(12.09,1.62)(12.09,0)
\psline(6.85,0)(13.71,0)
\psline(13.24,-4.38)(11.62,-4.38)
\psline(11.62,-4.38)(11.62,-6)
\psline(11.62,-6)(10,-6)
\psline(10,-6)(10,-7)
\psline(10,-7)(10.62,-7)
\psline(10.62,-7)(10.62,-8.62)
\psline(10.62,-8.62)(14.24,-8.62)
\psline(14.24,-8.62)(14.24,-7)
\psline(14.24,-7)(13.24,-7)
\psline(13.24,-7)(13.24,-4.38)
\psline(17,-4.38)(19.62,-4.38)
\psline(19.62,-4.38)(19.62,-1.76)
\psline(19.62,-1.76)(22.24,-1.76)
\psline(22.24,-1.76)(22.24,-6)
\psline(22.24,-6)(23.86,-6)
\psline(23.86,-6)(23.86,-8.62)
\psline(23.86,-8.62)(18,-8.62)
\psline(18,-8.62)(18,-6)
\psline(18,-6)(17,-6)
\psline(17,-6)(17,-4.38)
\psline(19.62,-8.62)(19.62,-7)
\psline(19.62,-7)(21.24,-7)
\psline(21.24,-7)(21.24,-6)
\psline(20.62,-6)(22.24,-5.99)
\psline(20.62,-4.38)(20.62,-6)
\psline(19.62,-4.38)(20.62,-4.38)
\psline{->}(4.91,1.81)(6.34,1.81)
\psline{->}(23.33,2.81)(25.02,2.81)
\psline{->}(14,-6)(16,-6)
\rput[tl](1.58,-0.19){\tiny{$D_1$}}
\rput[tl](19.88,-.19){\tiny{$D_2$}}
\rput[tl](11.79,-8.9){\tiny{$D_3$}}
\end{pspicture*}
\caption{Ammann A3 substitution.}
\end{figure}
More generally, any substitution for which one patch $S(D_i)$ contains another patch $S(D_j)$ can not be strongly recognizable.
The pentiamond substitution (figure 2; see the Tilings Encyclopedia at http://tilings.math.uni-bielefeld.de) is  strongly recognizable.
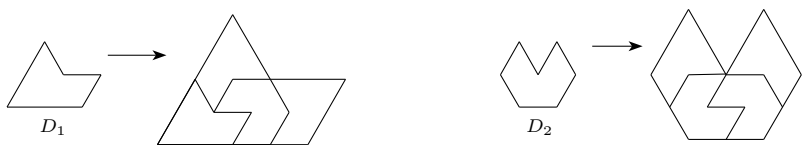
\begin{figure}[!htb]
\label{pentiamond}
\psset{xunit=1.0cm,yunit=1.0cm,algebraic=true,dotstyle=*,dotsize=3pt 0,linewidth=0.4pt,arrowsize=3pt 2,arrowinset=0.25}
\begin{pspicture*}(-1.72,0)(11.86,2)
\psline(2,0)(4,0)
\psline(2,0)(3,1.73)
\psline(4.5,0.87)(4,0)
\psline(3.5,0.87)(4.5,0.87)
\psline(3.5,0.87)(3,1.73)
\psline(2,0)(2.5,0.87)
\psline(2.75,0.43)(2.5,0.87)
\psline(2.75,0.43)(3.25,0.43)
\psline(3.25,0.43)(3,0)
\psline(3,0)(2,0)
\psline(2.75,0.43)(3,0.87)
\psline(3,0.87)(3.5,0.87)
\psline(3.5,0)(3.75,0.43)
\psline(3.5,0.87)(3.75,0.43)
\psline(3,0)(3.5,0)
\psline(0,0.5)(0.5,1.37)
\psline(0.75,0.93)(0.5,1.37)
\psline(0.75,0.93)(1.25,0.93)
\psline(1.25,0.93)(1,0.5)
\psline(1,0.5)(0,0.5)
\rput[tl](0.43,0.33){\tiny{$D_1$}}
\rput[tl](6.9,0.33){\tiny{$D_2$}}
\psline{->}(1.34,1.19)(2.12,1.19)
\psline(6.56,0.93)(6.81,1.37)
\psline(6.81,1.37)(7.06,0.93)
\psline(7.06,0.93)(7.31,1.37)
\psline(7.31,1.37)(7.56,0.93)
\psline(7.31,0.5)(7.56,0.93)
\psline(6.56,0.93)(6.81,0.5)
\psline(6.81,0.5)(7.31,0.5)
\psline(8.56,0.93)(9.06,1.8)
\psline(9.06,1.8)(9.56,0.93)
\psline(10.06,1.8)(9.56,0.93)
\psline(10.06,1.8)(10.56,0.93)
\psline(10.06,0.07)(10.56,0.93)
\psline(10.06,0.93)(10.31,0.5)
\psline(9.56,0.93)(10.06,0.93)
\psline(9.56,0.93)(9.31,0.5)
\psline(9.31,0.5)(9.81,0.5)
\psline(9.81,0.5)(9.56,0.07)
\psline(9.06,0.07)(10.06,0.07)
\psline(8.56,0.93)(9.06,0.07)
\psline(9.06,0.93)(8.81,0.5)
\psline(9.06,0.93)(9.56,0.94)
\psline{->}(7.78,1.31)(8.45,1.31)
\end{pspicture*}
\caption{Pentiamond substitution.}
\end{figure}
\end{eg}

 \section{Multi-Substitution Tiling Spaces}

In this section we establish the definition and basic properties of multi-substitution tiling spaces. These spaces are also referred to as \emph{mixed substitution} tiling spaces \cite{GM} or \emph{$S$-adic systems \cite{Du,Fe}}.  In \cite{FS}, the authors developed a framework   for studying the ergodic theory and topology
of hierarchical tilings in great generality. The classical substitution tiling spaces and the multi-substitution tiling spaces fit in this general framework. In fact, they are particular cases of \emph{fusion} tiling spaces. Certain properties, like minimality or
unique ergodicity, can be derived within the  framework of fusion tiling spaces. However, naturally, some other properties become hard to achieve in such generality.

\vspace{.20in}

 Let $\mathcal{S}=\{S_i\}_{i\in J}$ be a finite collection of substitutions $S_i:\mathcal{P}^1(X)\to\mathcal{P}(X)$.  Assume that the substitution $S_i$ can be extended to maps $S_i: \mathcal{P}(X)\to \mathcal{P}(X)$ and $S_i:X\to X$, for each $i\in J$.  Denote by $\lambda_i>1$ and $A_i$  the dilatation factor and the structure matrix, respectively, associated to $S_i$.   Provide the space of sequences $$\Sigma:=\{\bar a=(a_1,a_2,\ldots): {a_i}\in J\}$$ with the usual structure $(\Sigma,d_\Sigma)$ of metric space: given $\bar a=(a_1,a_2,\ldots)$ and $\bar b=(b_1,b_2,\ldots)$ in $\Sigma$, we set $d_\Sigma(\bar a,\bar b)=1/L$ if $L$ is the smallest integer such that $a_L\neq b_L$. We also introduce the standard shift map $\sigma:\Sigma\to\Sigma$, given by $\sigma(a_1,a_2,\ldots)=(a_2,a_3,\ldots)$, which is continuous with respect to $d_\Sigma$. Given $\bar a =(a_1,a_2,\ldots)\in\Sigma$, we denote by $[\bar a]_n$ the periodic sequence $(a_1,...,a_n,a_1,...,a_n,\ldots)\in\Sigma$. Clearly, for each $k>0$, $S_{\bar{a}}^k:=S_{a_1}\circ S_{a_2}\circ \ldots \circ S_{a_k}$ is itself a substitution with structure matrix given by  $A_{\bar{a}}^k:=A_{a_1}A_{a_2} \ldots A_{a_k}$ and dilation factor $\lambda_{\bar a}^n:=\lambda_{a_1}\lambda_{a_2}\ldots \lambda_{a_n}$.

 The sequence of substitutions $(S_{a_n})$ is called \emph{primitive} if for each $n$ there exists a least $N^{\bar a}_n$ such that $A_{a_{n}}A_{a_{n+1}}\ldots A_{a_{n+N^{\bar a}_n}}>0$. Observe that, in this case, each matrix $A_i$ does not have any column of all zeroes. Hence,
 $A_{a_{n}}A_{a_{n+1}}\ldots A_{a_{n+N^{\bar a}_n+j}}>0$ for all $j\geq 0$.
  If, for each $n$, the substitution $S_{a_n}$ is strongly primitive, that is, $A_{a_{n}}>0$, the sequence $(S_{a_n})$ is called \emph{strongly primitive}. The set $\mathcal{S}$ of substitutions is \emph{primitive} if, for any $\bar a\in \Sigma$, $(S_{a_n})$ is primitive. Moreover, we say that   a primitive set of substitutions $\mathcal{S}$ is \emph{bounded primitive} if the set $\{N^{\bar a}_n:\, n\in \mathbb{N}, \bar a\in \Sigma\}$ is bounded.
  \begin{lem}
    If $\mathcal{S}$ is primitive then it is bounded primitive.
  \end{lem}
  \begin{proof}
     The number of possible configurations of zero entries in finite products of structure matrices associated to substitutions in $\mathcal{S}$ is finite, say $L(\mathcal{S}).$ Now, assume that
      $\mathcal{S}$ is not bounded primitive. Then, for some $q>L(\mathcal{S})$, there is $\bar a\in \Sigma$ such that $A_{\bar a}^q$ has some zero entry. This means that we can find $p'<p\leq q$ such that  $A_{\bar a}^{p'}$ has the same zero configuration of entries as   $A_{\bar a}^p$. Then the sequence of substitutions $(S_{b_n})$, with $\bar b=(a_1,\ldots, a_{p'},a_{p'+1},\ldots,a_{p},a_{p'+1},\ldots,a_{p},a_{p'+1},\ldots)$ is non-primitive.
  \end{proof}

Before introduce the multi-substitution tiling spaces, let us prove the following useful lemma.

\begin{lem}\label{super}
Take  $D\in \mathcal{P}^1(X)$ and  assume that the sequence of substitutions $(S_{a_n})$ is primitive.  Given $n>0$, there are $R>0$ and $N_0>n$ such that, for any $N>N_0$ and any ball $B$ of radius $R$ contained in the support of $S^N_{\bar{a}}(D)$, the following holds: $\vec t + S^n_{\bar{a}}(D)\subset S^N_{\bar{a}}(D)$ and $\mathrm{supp}(\vec t + S^n_{\bar{a}}(D))\subset B$, for some $\vec t \in \rd$.
\end{lem}
\begin{proof}
 Take a finite number of  prototiles $\{D_1,\ldots,D_l\}$ of $X$ such that  $$\mathcal{T}^1(X)=\{[D_1],\ldots, [D_l]\}.$$ By primitivity, we can take ${N}_0$ such that a translated copy of $  S^n_{\bar{a}}(D)$ can be found in each $S^{{N}_0}_{\bar{a}}(D_i)$ for all $i\in\{1,\ldots,l\}$. Now, consider the tiles $\tilde{D}_i=\mathrm{supp}(S^{ N_0}_{\bar{a}}(D_i))$. For sufficiently large $R$, if $B$ is a ball with radius $R$ and $x'$ is a patch formed with translated copies of the tiles  $\tilde{D}_i$, with $B\subset \mathrm{supp}(x')$, then, for some $i\in\{1,\ldots,l\}$ and $\vec{t}_i\in \rd$, we have  $\vec{t}_i+\tilde{D}_i\subset B$ and $\vec{t}_i+\tilde{D}_i\in x'$. Take $N> N_0$ such that the support of $S^{N}_{\bar a}(D)$ contains some ball $B$ of radius $R$. Since $S^{N}_{\bar a}(D)$ is the disjoint union of translated copies of patches of the form $S^{N_0}_{\bar a}(D_i)$, the support of one of this copies must be contained in $B$, and we are done.
\end{proof}

\begin{rem}\label{remsei}
  It is clear from the proof that, given $n>0$ and tile $D$, if $\mathcal{S}$ is primitive (hence bounded primitive) we can take  $R>0$ and $N_0>n$ so that the statement of lemma \ref{super} holds for any $\bar a\in \Sigma$.
\end{rem}

Take  $D\in \mathcal{P}^1(X)$, $\bar{a}\in\Sigma$ and the corresponding sequence of patches in $\mathcal{P}(X)$: $$x_{\bar{a}}^1=D,\quad\textrm{ and }\quad x_{\bar{a}}^k=S_{\bar a}^k(D),\,\textrm{ for } k>1.$$  We define the \emph{multi-substitution tiling space}  $X_{\bar{a}}:=X_{\bar{a}}(\mathcal{S})\subseteq X$ as follows: a tiling $x\in X$ belongs to $X_{\bar{a}}$ if, and only if, for any finite patch $x'\subset x$ there exist $k>0$ and a vector $\vec{t}\in \rd$ such that $\vec{t}+x'\subseteq x_{\bar{a}}^k$.

 \begin{prop}\label{Xa} If the sequence of substitutions $(S_{a_n})$ is primitive, then:
  \begin{itemize}
    \item[a)] $X_{\bar{a}}\neq \emptyset$;
    \item[b)] $X_{\bar{a}}$ is independent of the initial tile $D\in\mathcal{P}^1(X)$;
    \item[c)] $X_{[\bar{a}]_n}$ coincides with the substitution tiling space $X_{S_{\bar a}^n}$;
    \item[d)] $X_{\bar{a}}$ is closed.
\end{itemize}
 \end{prop}
\begin{proof}
Take $D\in \mathcal{P}^1(X)$. By primitivity, the definition of  $X_{\bar{a}}$ is independent of the initial tile and, taking account lemma \ref{super}, for each $n>0$ there are  $\vec{t}_n\in\rd$, $k_n>0$ and $r_n>0$, with $\lim_n r_n=\infty$, such that
  the sequence $(x'_n)$ of patches $x'_n=\vec{t}_n+S_{\bar a}^{k_n}(D)$ satisfies: $x'_{n-1}\subset x'_n$ and $\mathrm{supp}(x'_{n-1})\subset  B_{r_n}\subset \mathrm{supp}(x'_n)$. Set
  \begin{equation}\label{x0}x_{\bar a}^\infty=\bigcup_{n\geq 1}x'_n\end{equation} and observe that $x_{\bar a}^\infty$ is a tiling of $\rd$ in $X_{\bar{a}}$. Hence $X_{\bar{a}}\neq \emptyset$.

It is clear that  $X_{S_{\bar a}^n}\subset X_{[\bar{a}]_n}$. Now, take $x\in X_{[\bar{a}]_n}$ and a finite patch $x'\subset x$. By definition, a translated copy of $x'$ appears in $S_{[\bar{a}]_n}^k(D)$ for some $k>0$. Take $N>k$ such that $A_{a_{k+1}}\ldots  A_{a_{N}}>0$ and $m$ such that $N\leq mn$. Then we also have $A_{a_{k+1}}\ldots  A_{a_{nm}}>0$. In particular, a translated copy of $D$ appears in $S_{a_{k+1}}\circ\ldots \circ S_{a_{nm}}(D)$. Consequently, a translated copy of $x'$ appears in $S_{[\bar{a}]_n}^{nm}(D)=(S_{[\bar{a}]_n}^{n})^m(D)$. This means that $x\in X_{S_{\bar a}^n}$.

  To prove that $X_{\bar{a}}$ is closed, take a sequence of tilings $(x_n)$ in  $X_{\bar{a}}$ converging to some $x\in X$. Take a patch $x''$ in $x$ and $r>0$ such that $\mathrm{supp}(x'')\subset B_{1/r}$. We know that there exists $n_0$ such that $d_T(x_n,x)<r$ for all $n>n_0$. This means that, for each $n>n_0$, there are patches $x'_n\in x_n[[B_{1/r}]]$ and $x'\in x[[B_{1/r}]]$, and a vector $\vec{t}_n$ with $\|\vec{t}_n\|<r$, such that $x'=\vec{t}_n+x'_n$. Since $x_n\in X_{\bar{a}}$, there exists a translation of $x'_n$, and consequently of $x''\subset x'$,  that is contained in some  $S_{\bar a}^{k_n}(D)$. Hence $x\in X_{\bar a}$.
 \end{proof}
Henceforth we assume that   the sequence of substitutions $(S_{a_n})$ is primitive. As  for substitution tiling spaces, recognizability is closely related with non-periodicity:
\begin{prop}
  Let  $\mathcal{S}=\{S_1,\ldots,S_k\}$ be a set of recognizable substitutions, $\bar a\in \Sigma$, and  $X_{\bar a}$  the corresponding multi-substitution tiling space. Then any tiling $x\in \xa$ is aperiodic. \end{prop}
  \begin{proof}The argument is standard.
  Take $x\in \xa$ and  suppose that $\vec t+x = x$ for some $\vec t\neq 0$. Since our substitutions are recognizable, for each $n\geq 1$ there exists a unique $x_n\in X_{\sigma^n (\bar{a})}$ such that $S_{\bar a}^n(x_n)=x$. We have
    $$S_{\bar a}^n(x_n)=\vec t+S_{\bar a}^n(x_n)=S_{\bar a}^n\Big(\frac{\vec t}{\lambda^n_{\bar a}}+x_n\Big),$$
    which means, by recognizability that
$x_n={\vec t}/{\lambda^n_{\bar a}}+x_n.$
    Now, for $n$ sufficiently large, it is clear that
    $\big\{\mathrm{supp}(D)+  {\vec t}/{\lambda^n_{\bar a}}\big\}\cap\mathrm{supp}(D)\neq \emptyset$
for any prototile $D$, which is a contradiction.

  \end{proof}

 It is well known that the set of periodic points of $\sigma$ is dense in $\Sigma$. Together with proposition \ref{Xa}, this result suggests  that any  multi-substitution tiling can be approximated arbitrarily  closely  by substitution tilings. In fact we have:

\begin{prop}For each $x\in X_{\bar{a}}$, there exists a sequence of tilings $(x_n)$, with $x_n\in X_{S^{j_n}_{\bar a}}$, for some $j_n\geq 1$, such that $\lim_n x_n=x$.
  \end{prop}
\begin{proof} For each $n$ take a patch $x'_n\in x[[B_{n}]]$. By definition of multi-substitution tiling space,  we have $x'_n\subset \vec t_n+ S_{\bar a}^{j_n}(D)$, for some $j_n\geq 1$ and $\vec t_n\in \mathbb{R}^d$. Adapting the procedure we have used in the proof of proposition \ref{Xa} to construct a tiling in $X_{\overline{a}}$,  it is possible to construct a tiling $x_n\in X_{S^{j_n}_{\bar a}}=X_{[\bar{a}]_{j_n}}$ containing  $\vec t_n+  S_{\bar a}^{j_n}(D)$. Clearly we have $\lim_n x_n=x$.
\end{proof}

\section{Minimality and Repetitivity}
As Sadun and Frank \cite{FS} have shown, fusion tiling spaces are minimal and its elements are repetitive. For completeness, we shall next give a proof of this result in the particular case of multi-substitution tiling spaces.
\vspace{.20in}

 A \emph{dynamical system} will be a pair $(Y,G)$ where $Y$ is a compact metric space and $G$ is a continuous action of a group. $(Y,G)$ is \emph{minimal} if $Y$ is the orbit closure $\overline{\mathcal{O}(y)}$ of any of its elements $y$. A point $y\in Y$ is \emph{almost periodic} if
 $$G(y,U)=\{g\in G:\,\,g(y) \in U\}$$ is \emph{relatively dense} (that is, there exists
a compact set $K\subseteq G$ such that $g\cdot K$ intersects $G(y,U)$ for all $g\in G$) for every open set $U\subseteq Y$ with $G(y,U)\neq \emptyset$.
Minimality and almost periodicity are related by Gottschalk's theorem:

\begin{thm}\cite{Go} Let  $(Y,G)$ be a dynamical system. If $y\in Y$ is an almost periodic point, then $(\overline{\mathcal{O}(y)}, G)$ is
minimal. Moreover, if $(Y,G)$ is minimal, then any point in $Y$ is  almost periodic.
 \end{thm}

 It is well known \cite{Rob} that, for any primitive substitution $S$, $(X_S,T)$ is minimal. More generally, for multi-substitutions tiling spaces we have:
 \begin{thm}
   The dynamical system $(X_{\bar{a}},T)$ is minimal.
 \end{thm}
 \begin{proof}
Let $x$ and $y$ be two tilings in $\xa$ and $\epsilon >0$. Fix $y'\in y[[B_{1/\epsilon}]]$. By definition of $\xa$, there is $n$ such that a translated copy of $y'$ can be found in $S_{\bar a}^n(D)$. Taking account lemma \ref{super}, there are $R>0$ and $N'>0$ such that, for any ball $B$ of radius $R$ with $B\subset \mathrm{supp}( S_{\bar a}^{N'}(D))$,  there is $\vec{t}$ for which $\vec{t}+S_{\bar a}^n(D)\subset S_{\bar a}^{N'}(D)$ and $\mathrm{supp}( \vec t + S_{\bar a}^{n}(D))\subset B $. On the other hand, again by definition of $\xa$, given a patch $x'\in x[[B_{R}]]$,  there is some $N''$ such that $S_{\bar a}^{N''}(D)$ contains a translated copy of $x'$. Hence, due to primitivity, we can take some $N\geq \max\{N',N''\}$ such that
$$y'\subseteq \vec t_1+S_{\bar a}^n(D)\subseteq \vec t_2+ x' \subseteq \vec t_3+ S_{\bar a}^{N}(D).$$ In particular, $d_T(\vec t_2+ x,y)<\epsilon$.
 \end{proof}
A tiling  $x$ of $\mathbb{R}^d$ is said to be \emph{repetitive}  if for every patch $x'$ of $x$ with bounded support there is some $r(x')>0$ such that, for every ball $B$ of $\mathbb{R}^d$ with radius $r(x')$, there exists $\vec{t}\in\mathbb{R}^d$ such that $\mathrm{supp}(\vec t+x' )\subseteq B$ and $\vec t+x' \subset x$. It is also common to refer to repetitive tilings as tilings satisfying the \emph{local isomorphism}  property \cite{Rad}. As explained in \cite{Rob}, for tiling dynamical systems $(X,T)$, repetitivity is equivalent to  {almost periodicity}.  From
Gottschalk's theorem it follows that:
\begin{thm}
  Any $x\in X_{\bar a}$ is repetitive.
\end{thm}
Of course, this can also be seen as an easy consequence of lemma \ref{super}. Given a tiling $x\in \xa$ and a finite patch $x'\subset x$, we have $x'\subset S_{\bar a}^n(D)$ for some $n$ and $D$. The radius $r(x')$ can be taken as the radius $R$ of lemma \ref{super}, which does not depend on the tiling $x$ of $\xa$ we take. On the other hand, $x'\in\mathcal{P}(X_{\bar b})$ for any $\bar b\in \Sigma$ with $d_\Sigma(\bar a, \bar b)<1/n$.
 Hence, taking account remark \ref{remsei}, we have:
 \begin{prop}\label{R}
 Assume that $\mathcal{S}$ is primitive and take $x'\subset S_{\bar a}^n(D)$. Then,  there exists $R>0$ such that,
 for any $\bar b\in \Sigma$ with $d_\Sigma(\bar a, \bar b)<1/n$,  any $x\in X_{\bar b}$ and any ball $B$ of radius $R$, a translated copy  $\vec t+x'$ of $x'$ can be founded in $x$ with $\mathrm{supp}(\vec t+x')\subseteq B$.
 \end{prop}

\section{Statistical stability}

\subsection{Unique ergodicity} The unique ergodicity of the system $(X_{\bar a}, T)$ was established in \cite{FS} in the general framework of fusion tiling spaces. Next we remake the proof of the unique ergodicity for multi-substitution tiling spaces, based on a result of Solomyak \cite{Sol}, and prove that
the convergence of patch frequencies to their ergodic limits is locally uniform in some open subset of $\Sigma$ (theorem \ref{cu}).

\vspace{.20in}

Given a patch $x'\in\mathcal{P}(\xa)$  and a measurable subset $U$ of $\rd$, define the \emph{cylinder set} $X^{\bar{a}}_{x',U}$ as
$$X^{\bar{a}}_{x',U}=\{x\in X_{\bar{a}}: \,\,\, x'+\vec{t}\subset x\,\,\mbox{for some $\vec{t}\in U$} \}.$$ These cylinders form a semi-algebra and a topology base for $\xa$.

For any set $H\subset \mathbb{R}^d$ and $r \geq 0$ we define
\begin{equation*}
H^{+r}=\{\vec t\in\rd:\,\mathrm{dist}(\vec{t},H)\leq r\},
\quad H^{-r}=\{\vec t\in H:\,\mathrm{dist}(\vec{t},\partial H)\geq r\},
\end{equation*}
 where $\partial H$  denotes the boundary of $H$.
A sequence $(H_n)$ of subsets of $\rd$ is a \emph{Van Hove} sequence if for any $r\geq 0$
$$\lim_n\frac{\mathrm{vol}((\partial H_n)^{+r})}{\mathrm{vol}(H_n)}=0.$$
 Clearly,  the sequence $(\mathrm{supp}(S_{\bar a}^n(D)))$ is a Van Hove sequence for each  $D\in\mathcal{P}^1(X)$.

 Consider the tiling $x_{\bar a}^\infty$ given by \eqref{x0}. For any patch $x'\in \mathcal{P}(\xa)$, denote by $L^{\bar a}_{x'}(H)$ (respectively, $N^{\bar a}_{x'}(H)$) the number of distinct translated copies of $x'$ in $x_{\bar a}^\infty$ whose support is completely contained in $H$ (respectively, intersects the border of $H$). If $y'\in \mathcal{P}(\xa)$ is another patch, we denote by  $L_{x'}(y')$ the number of distinct translated copies of $x'$ in $y'$ and by $\mathrm{vol}(y')$ the Euclidean volume of the support of $y'$.

\begin{thm}\label{freq vs van Hove}\cite{Sol}
  The  dynamical system $(\xa,T)$ is uniquely ergodic if for any patch $x'\in \mathcal{P}(\xa)$ there is a number $\mathrm{freq}_{\bar a}(x')>0$ such that, for any Van Hove sequence $(H_n)$,
  $$\freq_{\bar{a}} (x')=\lim_n \frac{L^{\bar a}_{x'}(H_n)}{\mathrm{vol}(H_n)}.$$ In this case, the unique ergodic measure $\mu_{\bar a}$ on $\xa$ satisfies
  $$\mu_{\bar a}(X^{\bar a}_{x',U})=\freq_{\bar{a}} (x')\mathrm{vol}(U)$$
  for all Borel subsets $U$ with $\mathrm{diam}(U)<\eta$, where $\eta>0$ is such that any prototile contains a ball of radius $\eta$.
\end{thm}

For a primitive substitution tiling space $X_{S}$,  such \emph{frequencies} exist \cite{GH,Sol}.  For example, if $\vec p=(p_1,\dots,p_l)$ and $\vec q=(V_1,\ldots,V_l)$ are  right and left PF-eigenvectors, respectively, of $A_S$ satisfying $\vec{p}\cdot \vec{q}=1$, then
$\freq_{S}{(D_i)}=p_i$.
Consequently,  $(X_S,T)$ is uniquely ergodic. We want to extend this result for multi-substitutions tiling spaces.

So, start with a primitive set  $\mathcal{S}=\{S_1,\ldots, S_k\}$ of substitutions acting on the set of prototiles $\{D_1,\ldots, D_l\}$.
 Being primitive, there exists some $N(\mathcal{S})>0$ such that $A>0$  for any $A\in\mathcal{A}_\mathcal{S}^{N(\mathcal{S})}$, the set of all $N(\mathcal{S})$-products of matrices from $\{\frac{A_1}{\omega_1},\ldots ,\frac{A_k}{\omega_k}\},$ where $\omega_i$ is the PF-eigenvalue of $A_i$.

For each $i,j\in\{1,\ldots,l\}$, set
  \begin{equation}\label{eqcDi}c_{\bar{a}}(D_i):=\lim_n \frac{L_{D_i}(S_{\bar a}^n(D_j))}{\mathrm{vol}(S_{\bar a}^n(D_j)))}=\lim_n\frac{(A_{\bar a}^n)_{ij}}{\omega_{\bar a}^n\mathrm{vol}(D_j)}.\end{equation}
\begin{lem}\label{cDi}
 The limit \eqref{eqcDi} exists, does not depend on $j$ and is uniform with respect to $\bar{a}\in \Sigma$. Moreover, there are $C>0$ and $0<\theta<1$ such that, for any $\delta>0$  and   $\bar a, \bar b\in\Sigma$ with $d_\Sigma(\bar a,\bar b)<\delta$, we have
\begin{equation}\label{contfreq}
|c_{\bar a}(D_i)-c_{\bar{b}}(D_i)|<C\theta^{1/\delta},\end{equation}
for all $i\in\{1,\ldots,l\}$.
\end{lem}
\begin{proof}Let $\bar a\in \Sigma$. For each $n\geq 1$, consider the matrix $E_{\bar a}^n=[(E_{\bar a}^n)_{ij}]$ defined by
$$(E_{\bar a}^n)_{ij}=\frac{(A_{\bar a}^n)_{ij}}{\omega_{\bar a}^n\mathrm{vol}(D_j)}.$$
Let $\Delta_{\bar a}^n\subset \rd$ be the convex hull of the columns of $E_{\bar a}^n$. It is easy to check that each column of $E_{\bar a}^{n+N}$ sits in $\Delta_{\bar a}^n$. Indeed,
$$(E_{\bar a}^{n+N})_{ij}=\sum_{k=1}^l \frac{(A_{\bar a}^n)_{ik}}{\omega_{\bar a}^n\mathrm{vol}(D_k)}\frac{(A_{\sigma^{n}(\bar a)}^N)_{kj}\mathrm{vol}(D_k)}{\omega_{\sigma^{n}(\bar a)}^N\mathrm{vol}(D_j)}.$$
Hence, the $j$-column $v_{\bar a,j}^{n+N}$ of $E_{\bar a}^{n+N}$  is given by
\begin{equation}\label{vs}
v_{\bar a,j}^{n+N}=\sum_{k=1}^l v_{\bar a,k}^{n}\frac{(A_{\sigma^n(\bar a)}^N)_{kj}\mathrm{vol}(D_k)}{\omega_{\sigma^{n}(\bar a)}^N\mathrm{vol}(D_j)}.\end{equation}
Since $$\sum_{k=1}^l\frac{(A_{\sigma^n(\bar a)}^N)_{kj}\mathrm{vol}(D_k)}{\omega_{\sigma^{n}(\bar a)}^N\mathrm{vol}(D_j)}=1,$$
we have $v_{\bar a,j}^{n+N}\in \Delta_{\bar a}^n$, that is  $\Delta_{\bar a}^{{n+N}}\subseteq \Delta_{\bar a}^{{n}}$ for all $N>0$. Set $\Delta_{\bar a}=\bigcap_{n=1}^\infty \Delta_{\bar a}^{{n}}$ and take $\theta'$ (not depending on $\bar a$) satisfying $\theta' l<1$ and
$$0<\theta'<\min _{A\in\mathcal{A}_\mathcal{S}^{N(\mathcal{S})}}\min_{k,j} \Big\{\frac{A_{kj}\mathrm{vol}(D_k)}{\mathrm{vol}(D_j)}\Big\}.$$
Taking account \eqref{vs}, we have
\begin{align*}
  v_{\bar a,j}^{n+N(\mathcal{S})}&=\sum_{k=1}^l (1-\theta'l)v_{\bar a,k}^{n}\Big(\frac{(A_{\sigma^n(\bar a)}^{N(\mathcal{S})})_{kj}\mathrm{vol}(D_k)}{(1-\theta'l)\omega_{\sigma^n(\bar a)}^{N(\mathcal{S})}\mathrm{vol}(D_j)}-\frac{\theta'}{1-\theta'l}\Big)+ \sum_{k=1}^l v_{\bar a,k}^{n}\theta'.
\end{align*}
Observe that
$$\sum_{k=1}^l\Big( \frac{(A_{\sigma^n(\bar a)}^{N(\mathcal{S})})_{kj}\mathrm{vol}(D_k)}{(1-\theta'l)\omega_{\sigma^n(\bar a)}^{N(\mathcal{S})}\mathrm{vol}(D_j)}-\frac{\theta'}{1-\theta'l}\Big)=1.$$
Hence $\Delta_{\bar a}^{n+N(\mathcal{S})}$ is contained, up to translation, in $(1-\theta'l)\Delta_{\bar a}^n$. Then
\begin{equation}\label{diam}
\mathrm{diam}(\Delta_{\bar a}^{1+nN(\mathcal{S})})\leq \rho(1-\theta'l)^n
\end{equation} for all $n$, where $\rho$ is the maximum of the $l$ possible values of $\mathrm{diam}(\Delta_{\bar a}^1)$.
  Consequently $\Delta_{\bar a}$ has diameter zero, hence it consists of a single point, which means that the limit \eqref{eqcDi} exists and does not depend on $j$. Moreover, this limit is  uniform with respect to $\bar{a}\in \Sigma$ since the majoration \eqref{diam} does not depend on $\bar a$.

Finally, note that if $d_{\Sigma}(\bar a,\bar b)=1/k<\delta$, that is $a_j=b_j$ and $\Delta_j:=\Delta_{\bar a}^j=\Delta_{\bar b}^j$ for all $j\in\{1,\ldots,k-1\}$, then
$$|c_{\bar a}(D_i)-c_{\bar b}(D_i) |\leq \mathrm{diam}(\Delta_{k-1})\leq C\theta^{1/\delta}$$
for  $\theta=(1-\theta'l)^{1/N(\mathcal{S})}$ and some constant $C>0$.
\end{proof}

Take a finite patch $x'\in\mathcal{P}(X)$.  For $\bar a\in \Sigma$, set
\begin{equation}\label{cx}c_{\bar a}({x'}):=\lim_n \frac{L_{x'}(S_{\bar a}^n(D_j))}{\mathrm{vol}(S_{\bar a}^n(D_j)))}.\end{equation}

\begin{lem}\label{exists freq P}
The limit \eqref{cx} exists and does not depend on $j$. Moreover, if $x'\in\mathcal{P}(X_{\bar a})$, then the limit is uniform in a small neighborhood of $\bar a$.
\end{lem}

\begin{proof}
Take a finite patch $x'\in\mathcal{P}(X)$ and a sequence $\bar a\in \Sigma$.
Given a subset $H$ of $\rd$, observe that
\begin{equation}\label{vol}
\frac{L^{\bar a}_{x'}(H)}{\mathrm{vol}(H)}\leq \frac{1}{{\mathrm{vol}(x')}}.
\end{equation}

  Given $n>m$,
write
$$S_{\sigma^m(\bar a)}^{n-m}(D_j)=\bigcup_{i,k} D_{ik}\,,$$ where $i\in \{1,\ldots ,l\}$, $k\in \{1,\ldots, \big(A_{\sigma^m(\bar a)}^{n-m}\big)_{ij}\}$ and each
 $D_{ik}$ is a translated copy of the prototile $D_i$. Denote by $N_{x'}^{\bar a}(n,m,i,j)$ the number of translated copies of $x'$ contained in $S_{\bar a}^n(D_j)$ whose support  intersects the boundary of some $\mathrm{supp}(S_{\bar a}^m(D_{ik}))$.
Since the sequence $(S_{\bar{a}}^m(D_i))$ is Van Hove, for each $\epsilon >0$ there exists $m(\bar a)$ such that
\begin{equation*}\label{Nx}
N_{x'}^{\bar a}(n,m,i,j)<\epsilon L_{x'}(S_{\bar{a}}^m(D_i))\big(A_{\sigma^m(\bar a)}^{n-m}\big)_{ij},
\end{equation*}
for all $m>m(\bar a)$ and $i\in\{1,\ldots,l\}$. Hence
\begin{equation}\label{conf}
\sum_{i=1}^l\frac{L_{x'}(S_{\bar{a}}^m(D_i))\big(A_{\sigma^m(\bar a)}^{n-m}\big)_{ij}}{\omega_{\bar a}^m\mathrm{vol}(S_{\sigma^m(\bar a)}^{n-m}(D_j))}\leq \frac{L_{x'}(S_{\bar a}^n(D_j))}{\mathrm{vol}(S_{\bar a}^n(D_j)))}\leq  (1+\epsilon)\sum_{i=1}^l\frac{L_{x'}(S_{\bar{a}}^m(D_i))\big(A_{\sigma^m(\bar a)}^{n-m}\big)_{ij}}{\omega_{\bar a}^m\mathrm{vol}(S_{\sigma^m(\bar a)}^{n-m}(D_j))}.
\end{equation}
Taking account \eqref{vol} and lemma \ref{cDi}, we have
\begin{align*}
  \limsup_n\frac{L_{x'}(S_{\bar a}^n(D_j))}{\mathrm{vol}(S_{\bar a}^n(D_j))}-\liminf_n\frac{L_{x'}(S_{\bar a}^n(D_j))}{\mathrm{vol}(S_{\bar a}^n(D_j))}&\leq \frac{\epsilon}{\mathrm{vol} (x')}.
\end{align*}
Since this holds for an arbitrary $\epsilon$, it follows that the limit $c_{\bar a}({x'})$ exists. The independence of $c_{\bar a}(x')$ with respect to $j$ follows from the observation that
 the lower and upper bounds in \eqref{conf} do not depend on $j$ when $n$ goes to infinity.

Suppose now that $x'\in\mathcal{P}(X_{\bar a})$. Let us prove that the limit is uniform in a small neighborhood of $\bar a$. We have $x'\subset S_{\bar a}^{n_0}(D)$ for some $n_0$ and $D$.
 Take the radius $R>0$ associated to $x'$ given by proposition \ref{R}. For $\lambda>0$, let $L(R,\lambda,i)$ be the maximum number of open disjoint balls of radius $R$ contained in $\lambda D_i$.  Clearly, there exists $\lambda_0$ such that
$$\frac{\mathrm{vol}((\partial{\lambda D_i)}^{+r})}{\mathrm{vol}(x')}   <\epsilon L(R,\lambda,i)$$
for all $\lambda>\lambda_0$, with $r$ the diameter of $\mathrm{supp}(x')$. Let $\omega=\min_i\{\omega_i\}$ and take $m_0>0$ such that $\omega^{m_0}>\lambda_0$. Then, for all $m>m_0$ and $\bar b\in \Sigma$ with $d_\Sigma(\bar a, \bar b)<1/n_0$, we have $x'\in\mathcal{P}(X_{\bar b})$ and
\begin{align*}
N_{x'}^{\bar b}(n,m,i,j)&< \frac{\mathrm{vol}\big((\partial{\omega_{\bar b}^m D_i)}^{+r}\big)}{\mathrm{vol}(x')}(A_{\sigma^m(\bar b)}^{n-m}\big)_{ij}\\&<\epsilon L(R,\lambda,i) (A_{\sigma^m(\bar b)}^{n-m}\big)_{ij} \\&<\epsilon L_{x'}(S_{\bar{b}}^m(D_i))\big(A_{\sigma^m(\bar b)}^{n-m}\big)_{ij},
\end{align*}
and we are done.
\end{proof}

\begin{lem}\label{freq VH implies uni erg}
Assume that the limit \eqref{cx} exists. Then,
  for  any Van Hove sequence  $(H_n)$,
  $$c_{\bar a}(x')=\lim_n \frac{L^{\bar a}_{x'}(H_n)}{\mathrm{vol}(H_n)}.$$
\end{lem}

\begin{proof}Although the proof follows closely that of \cite{LMS} for substitution tilings, we present it here since it is essential to establish theorem \ref{cu}.
For each $m>0$, there is $\mathcal{D}^{\bar a}_m\subseteq\mathcal{P}^1(X_{\bar a})$ such that the tiling $x_{\bar a}^\infty$ in \eqref{x0} is the disjoint union of patches  $S_{\bar a}^{m}(D)$, with $D\in \mathcal{D}^{\bar a}_m$.  Given $m>0$ and $n>0$, define
\begin{align*}G^{\bar a}_{m,n}:=\{D\in\mathcal{D}^{\bar a}_m: \mathrm{supp}(S_{\bar a}^m(D))\cap H_n \neq\emptyset\},\,\,\, H^{\bar a}_{m,n}:=\{D\in\mathcal{D}^{\bar a}_m: \mathrm{supp}(S_{\bar a}^m(D))\subseteq H_n \}.
\end{align*}
Hence
\begin{equation}\label{ineq aux freq v hove 0}\sum_{D\in H^{\bar a}_{m,n}}L_{x'}(S_{\bar a}^m(D))\leq L^{\bar a}_{x'}(H_n)\leq \sum_{D\in G^{\bar a}_{m,n}} \Big(L_{x'}(S_{\bar a}^m(D))+N^{\bar a}_{x'}(\partial S_{\bar a}^m(D))\Big),\end{equation}
where $\partial S_{\bar a}^m(D)$ denotes the border of the support of $S_{\bar a}^m(D)$.
Now, fix $\epsilon>0$. Taking account that $(S_{\bar a}^m(D_j))$ is a Van Hove sequence, we can take $m$ large enough so that, for every $D\in \mathcal{D}^{\bar a}_m$, we have
\begin{equation}\label{unif}
\left|\frac{L_{x'}( S_{\bar a}^m(D))}{\textrm{vol}(S_{\bar a}^m(D))}-c_{\bar a}(x')\right|<\epsilon\quad\textrm{and}\quad
N_{x'}(\partial S_{\bar a}^m(D))<\epsilon L_{x'}( S_{\bar a}^m(D)).\end{equation}
Together with \eqref{ineq aux freq v hove 0}, this gives
\begin{equation}\label{ineq aux freq v hove}(c_{\bar a}(x')-\epsilon)\sum_{D\in H^{\bar a}_{m,n}}\textrm{vol}(S_{\bar a}^m(D))\leq L^{\bar a}_{x'}(H_n)\leq(1+\epsilon)(c_{\bar a}(x')+\epsilon) \sum_{D\in G^{\bar a}_{m,n}}\textrm{vol}(S_{\bar a}^m(D)).\end{equation}
On the other hand, note that, setting $t_m:=\max_{j}\{\textrm{diam}(S_{\bar a}^m(D_j)\}$, since $m$ is fixed, for large enough $n$ we have
$$\sum_{D\in H^{\bar a}_{m,n}}\textrm{vol}(S_{\bar a}^m(D))\geq \textrm{vol}(H_n^{-t_m})\geq (1-\epsilon)\textrm{vol}(H_n)$$
and
$$\sum_{D\in G^{\bar a}_{m,n}}\textrm{vol}(S_{\bar a}^m(D))\leq \textrm{vol}(H_n^{+t_m})\leq (1+\epsilon)\textrm{vol}(H_n),$$
which, combining with  \eqref{ineq aux freq v hove}, concludes the proof, since $\epsilon>0$ is arbitrary.
\end{proof}

Combining theorem \ref{freq vs van Hove} with lemmas \ref{exists freq P}  and \ref{freq VH implies uni erg}, with $c_{\bar a}(x')=\mathrm{freq}_{\bar a}(x')$, we conclude that
\begin{thm}\label{unique erg}
 If $\mathcal{S}$ is primitive, $(\xa,T)$ is uniquely ergodic for all $\bar a\in \Sigma$.
\end{thm}

The following theorem establishes that the patch frequencies converge uniformly to their ergodic limits in an open subset of $\Sigma$.

\begin{thm}\label{cu}
 If $x'\in\mathcal{P}(X_{\bar a})$ and $H_n$ is a Van Hove sequence, then the sequences $$\frac{L^{\bullet}_{x'}(H_n)}{\mathrm{vol}(H_n)}$$ converge uniformly to $\mathrm{freq}_{\bullet}(x')$ in a small neighborhood of $\bar a$.
\end{thm}
\begin{proof}
This follows easily from lemma \ref{exists freq P} and from the proof of lemma \ref{freq VH implies uni erg}, since  \eqref{unif} holds for any sequence in a small neighborhood of $\bar a$.
\end{proof}

We denote by $\mu_{\bar a}:=\mu_{\bar a,\mathcal{S}}$ the unique ergodic measure of $(\xa,T)$. Two ergodic dynamical systems $(X,G,\mu)$ and $(Y,H,\nu)$ are said to be  \emph{isomorphic} if there exist a group isomorphism $\xi:G\to H$ and a measure preserving homeomorphism $F:X\to Y$ such that $F\circ g=\xi(g)\circ F$ for all $g\in G$.

\begin{thm}
   Let  $\mathcal{S}=\{S_1,\ldots,S_k\}$ be a set of strong recognizable substitutions.
   Then the ergodic dynamical systems  $(X_{\bar a},T,\mu_{\bar a})$ and $(X_{\sigma(\bar a)},T,\mu_{\sigma(\bar a)})$ are isomorphic.
   \end{thm}
\begin{proof}
First note that $S_{a_1}\circ \vec t=\lambda_{a_1}\vec t\circ S_{a_1}$ and  $\xi:T\to T$ defined by $\xi(\vec t\,)=\lambda_{a_1}\vec t$ is an isomorphism.
By  recognizability, $S_{a_1}$ is a bijection from $X_{\sigma(\bar a)}$ onto $\xa$. On the other hand, if $x,y\in X_{\sigma(\bar a)}$ and $d(x,y)< \epsilon$, then $d(S_{a_1}(x),S_{a_1}(y))<\lambda_{a_1}\epsilon$; similarly, if $x,y\in X_{\bar a}$ and $d(x,y)< \epsilon$, then $d(S^{-1}_{a_1}(x),S^{-1}_{a_1}(y))<\lambda_{a_1}\epsilon$. Hence $S_{a_1}:X_{\sigma(\bar a)}\to \xa$ is a homeomorphism. To prove that $S_{a_1}$ is measure preserving is sufficient to prove it for cylinders $X_{P,U}$ with $U$ sufficiently small.
By strong recognizability, it is clear that $L_{S_{a_1}(P)}(S_{\bar a}^n(D_j))=L_{P}(S_{\sigma(\bar a)}^{n-1}(D_j))$. Then
\begin{align*}
 \nonumber \mu_{\bar a}\big(S_{a_1}(X^{\sigma(\bar a)}_{P,U})\big)&=\mu_{\bar a}\big(X^{\bar a}_{S_{a_1}(P),\lambda_{a_1}U}\big)=\mathrm{freq}_{\bar a}\big(S_{a_1}(P)\big)\mathrm{vol}(\lambda_{a_1}U)\\
  &=\frac{1}{\lambda_{a_1}^d}\mathrm{freq}_{\sigma(\bar a)}(P)\lambda_{a_1}^d\mathrm{vol}(U)=\mu_{\sigma(\bar a)}\big(X^{\sigma(\bar a)}_{P,U}\big).
\end{align*}
\end{proof}

\subsection{Statistical stability}
The inequality \eqref{contfreq} says, in particular, that  $\bar a\mapsto c_{\bar a}(D)$ defines a continuous map  $\Sigma\to \mathbb{R}$ for each tile $D$. In this subsection we extend this result to arbitrary patches $x'$. As a consequence, we will see that the unique measures $\mu_{\bar a}$, although  defined in different spaces, also satisfy a certain kind of continuity with respect to $\bar a\in \Sigma$.

\vspace{.20in}

Recall that the \emph{upper Minkowski dimension}  of a subset $H\subset \mathbb{R}^d$ can be defined as
$$\mathcal{D}_H:=\inf\{\beta:\,\mathrm{vol}(H^{+r})=O(r^{d-\beta})\,\,\mbox{as $r\to 0^+$}\}.$$ Set $d-1\leq\mathcal{D}:=\max_i\{\mathcal{D}_{\partial D_i}\}<d$.

\begin{thm}\label{ss} Given $N>0$, there is $C_N>0$  such that, for any $\delta>0$  and any $\bar a, \bar b\in\Sigma$ with $d_\Sigma(\bar a,\bar b)<\delta$, we have, for all $x'\in\mathcal{P}^N(X)$,
$$|c_{\bar a}(x')-c_{\bar{b}}(x')|<C_N\theta_0^{1/\delta},$$ where $\theta_0=\theta^{\frac{\mathcal{D}-d}{\mathcal{D}-d+\log_\omega\theta}}<1$ and $\theta$ is given by lemma \ref{cDi}.
\end{thm}

\begin{proof}

Set $\omega=\min_i\{\omega_i\}$, $t=\max_{j}\{\textrm{diam}(D_j)\}$ and $v=\min_{j}\{\mathrm{vol}(D_j)\}$. We have
$$\frac{\mathrm{vol}\big((\partial S_{\bar{a}}^m(D_i))^{+t}\big)}{\mathrm{vol}(S_{\bar{a}}^m(D_i))}\leq \frac{\mathrm{vol}\big((\partial D_i)^{+\frac{t}{\omega^m}}\big)}{\mathrm{vol}(D_i)}=O\Big({1}/{\omega^{m(d-\mathcal{D})}}\Big)$$ as $m\to \infty$, for all $\bar a\in \Sigma$.
Hence, for some constants $C_1$ and $m_1$,

\begin{align*}
  N_{x'}^{\bar a}(n,m,i,j)&<\frac{N\mathrm{vol}((\partial S^m_{\bar a}(D_i))^{+t})\big(A_{\sigma^m(\bar a)}^{n-m}\big)_{ij}}{v} < \frac{C_1N}{v}\frac{\mathrm{vol}(S^m_{\bar a}(D_i))
  \big(A_{\sigma^m(\bar a)}^{n-m}\big)_{ij}}{\omega^{m(d-\mathcal{D})}}
  \end{align*}
for all $\bar a\in \Sigma$ and $n>m>m_1$.
On the other hand,
$$\sum_{i=1}^l  \frac{\mathrm{vol}(S^m_{\bar a}(D_i))\big(A_{\sigma^m(\bar a)}^{n-m}\big)_{ij}}{\mathrm{vol}(S^n_{\bar a}(D_j))}=1$$
for all $j$. Hence
\begin{align*}
\sum_{i=1}^l&\frac{L_{x'}(S_{\bar{a}}^m(D_i))\big(A_{\sigma^m(\bar a)}^{n-m}\big)_{ij}}{\omega_{\bar a}^m\mathrm{vol}(S_{\sigma^m(\bar a)}^{n-m}(D_j))}\leq \frac{L_{x'}(S_{\bar a}^n(D_j))}{\mathrm{vol}(S_{\bar a}^n(D_j)))} \leq  \sum_{i=1}^l\frac{L_{x'}(S_{\bar{a}}^m(D_i))\big(A_{\sigma^m(\bar a)}^{n-m}\big)_{ij}}{\omega_{\bar a}^m\mathrm{vol}(S_{\sigma^m(\bar a)}^{n-m}(D_j))}+\frac{C_1N}{v\omega^{m(d-\mathcal{D})}}.
\end{align*}
Taking the limit $n\to\infty$   we obtain
\begin{equation}\label{conf2}
\sum_{i=1}^l\frac{L_{x'}(S_{\bar a}^m(D_i))}{\omega_{\bar a}^m}c_{\sigma^m(\bar a)}(D_i)       \leq c_{\bar a}({x'})\leq \sum_{i=1}^l\frac{L_{x'}(S_{\bar a}^m(D_i))}{\omega_{\bar a}^m}c_{\sigma^m(\bar a)}(D_i)+\frac{C_1N}{v\omega^{m(d-\mathcal{D})}},
\end{equation}
for all $m>m_1$  and  $\bar a\in \Sigma$. Set
\begin{equation}\label{gamma}
\gamma := 1-\frac{d-\mathcal{D}}{\log_\omega\theta}>1,
\end{equation}
with $\theta<1$ given by lemma \ref{cDi}. Observe also that
\begin{equation}\label{qqq}
  L_{x'}(S_{\bar a}^m(D_i))\leq \frac{\mathrm{vol}(S_{\bar a}^m(D_i))}{\mathrm{vol}(x')}\leq  \frac{\mathrm{vol}(S_{\bar a}^m(D_i))}{Nv}
\end{equation}
In view of \eqref{conf2}, \eqref{gamma}, \eqref{qqq}  and lemma \ref{cDi}, there is a constant $C$ such that, if, for some $m>0$,
 $d_\Sigma(\bar a,\bar b)<\frac{1}{\gamma m}=\delta$,
\begin{align*}
|c_{\bar a}(x')-  c_{\bar b}(x')|&\leq \sum_{i=1}^l\frac{L_{x'}(S_{\bar a}^m(D_i))}{\omega_{\bar a}^m}|c_{\sigma^m(\bar a)}(D_i)-c_{\sigma^m(\bar b)}(D_i)|+\frac{C_1N}{v\omega^{m(d-\mathcal{D})}} \\ &\leq\sum_{i=1}^l\frac{\mathrm{vol}(D_i)}{Nv}C'\theta^{(\gamma-1)m}+ \frac{C_1N}{v\omega^{m(d-\mathcal{D})}} \\ & \leq
C_N\big(\theta^{\frac{\mathcal{D}-d}{\gamma \log_\omega \theta}}\big)^{1/\delta}\end{align*}
for all $x'\in\mathcal{P}^N(X)$.
  \end{proof}

From theorem \ref{freq vs van Hove} and theorem \ref{ss} we have:
\begin{cor}\label{stat}
Given $N>0$, there is $C_N>0$  such that, for any $\delta>0$  and any $\bar a, \bar b\in\Sigma$ with $d_\Sigma(\bar a,\bar b)<\delta$, we have, for all $x'\in\mathcal{P}^N(X)$ and all sufficiently small Borel subset $U\subset \mathbb{R}^d$,
$$|\mu_{\bar a}(X^{\bar a}_{x',U})-\mu_{\bar b}(X^{\bar b}_{x',U})|<C_N\theta_0 ^{1/\delta},$$
where $\theta_0$ is given by theorem \ref{ss}.
\end{cor}

To finalize, we extend the previous corollary:

\begin{thm}\label{int} Take finite patches $P,Q\in\mathcal{P}(X)$  and measurable sets $U,V\subset \rd$.  There exists $C>0$ such that, for any $\delta>0$ sufficiently small and any $\bar a, \bar b\in\Sigma$ with $d_{\Sigma}(\bar a, \bar b)<\delta$, we have
  \begin{equation*}
  |\mu_{\bar a}(X^{\bar a}_{ P,U}\cap X^{\bar a}_{Q,V})-\mu_{\bar b}(X^{\bar b}_{ P,U}\cap X^{\bar b}_{Q,V})|<C\theta_0^{\frac{1}{\delta}},\end{equation*}
 where $\theta_0$ is given by theorem \ref{ss}.
  \end{thm}

\begin{proof} Let $B(P,U)$ and $B(Q,V)$ be two open balls such that  $\mathrm{supp}(\vec u+P)\subset B(P,U)$ and $\mathrm{supp}(Q+\vec v) \subseteq B(Q,V)$ for all $\vec u\in U$ and $\vec v\in V$.
Consider the set $\mathcal{P}(P,Q,{\bar a})$ of finite patches $x'$ of $\xa$ such that $B:=B(P,U)\cup B(Q,V)$ is contained in the interior of $\mathrm{supp}(x')$ and all tiles of $x'$ intersect  $B$. There are finitely many equivalence classes $\{[x_1'],\ldots,[x'_{k}]\}$ of such patches, with $x'_i\in\mathcal{P}(P,Q,{\bar a})$. For each $i\in I:=\{1,\ldots,k\}$, set $$U_i=\{\vec w:\,\, \vec w+x'_i\in \mathcal{P}(P,Q,{\bar a})\},$$ which is an open subset of $\mathbb{R}^d$. Then we have a   disjoint cylinder decomposition
$$X_{\bar a}=\bigcup_{i\in I}X^{\bar a}_{x'_i,U_i}.$$ The diameter of $U_i$ is smaller then $\max_j\{\mathrm{diam}(D_j)\}$. Take $N_0,\psi >0$ such that, for all $\bar a$ and $i$, the following hold: $\mathrm{vol}(U_i)\leq \psi $; if $x'_i\in \mathcal{P}^N(X)$  then $N\leq N_0$.  For simplicity of exposition we assume that $\mathrm{diam}(U_i)<\eta$ where $\eta$ is such that any tile contains a ball of radius $\eta$. Otherwise, we could decompose each $U_i$ into a fixed number of disjoint subsets satisfying this property.

Consider the subset $W_i\subset U_i$ of vectors $\vec w$ such that $\vec u+P$ and $\vec v+Q$ are contained in $\vec w+x'_i$ for some $\vec u\in U$ and $\vec v\in V$. We have
$$X^{\bar a}_{P,U}\cap X^{\bar a}_{Q,V}\cap X^{\bar a}_{x'_i,U_i}=X^{\bar{a}}_{x'_i,W_i}.$$
Then
\begin{equation}\label{mua}
\mu_{\bar a}(X^{\bar a}_{P,U}\cap X^{\bar a}_{Q,V})=\!\!\sum_{i\in I}\mu_{\bar a}(X^{\bar a}_{x'_i,W_{i}}).\end{equation}

Take $\bar b\in \Sigma$ with $d_{\Sigma}(\bar a, \bar b)<\delta$.
We have a finite disjoint cylinder decomposition
\begin{equation*}\label{decomcylindersb}
X^{\bar b}_{ P,U}\cap X^{\bar b}_{Q,V}=\!\!\bigcup_{i\in I,j\in \hat I}\!\!\!X^{\bar b}_{x'_i,W_i}\cup X^{\bar b}_{\hat x'_j,\hat W_j},
\end{equation*}
where, for each $i\in \hat I$, $\hat x'_i$ is a patch in $\mathcal{P}(P,Q,{\bar b})$ but not in $\mathcal{P}(P,Q,{\bar a})$.
Hence
\begin{equation}\label{mub}
\mu_{\bar b}(X^{\bar b}_{P,U}\cap X^{\bar b}_{Q,V})= \!\!\sum_{i\in I} \mu_{\bar b}(X^{\bar b}_{x'_i,W_i})+\!\!\sum_{i\in \hat{I}}\mu_{\bar b}(X^{\bar b}_{\hat{x}'_i,\hat W_i}).\end{equation}
Equations
 \eqref{mua} and \eqref{mub} give
\begin{align}
\nonumber |\mu_{\bar b}(X^{\bar b}_{P,U}\cap X^{\bar b}_{Q,V})-&\mu_{\bar a}(X^{\bar a}_{P,U}\cap X^{\bar a}_{Q,V})| \\ &\leq
 \sum_{i\in I}|\mu_{\bar b}(X^{\bar b}_{x'_i,W_i})-\mu_{\bar a}(X^{\bar a}_{x'_i,W_i})|+\!\!\sum_{i\in \hat{I}} \mu_{\bar b}(X^{\bar b}_{\hat{x}'_i,\hat W_i{}}).\label{a1}
\end{align}
But, taking account theorem \ref{ss}, we have
\begin{align}\nonumber
 \sum_{i\in I}|\mu_{\bar b}(X^{\bar b}_{x'_i, W_i})-\mu_{\bar a}(X^{\bar a}_{x'_i,W_i})|&= \sum_{i\in I}|c_{\bar b}(x'_i)-
c_{\bar a}(x'_i)|\mathrm{vol}(W_i)\\ &\leq
 C| I| \psi\theta_0^{\frac{1}{\delta}},\label{a2}
\end{align}
for some constant $C$. On the other hand,  it follows from
\begin{equation*}
\sum_{i\in {I}} \mu_{\bar a}(X^{\bar a}_{{x}'_i, U_{i}}) =   \!\sum_{i\in {I}} \mu_{\bar b}(X^{\bar b}_{{x}'_i, U_{i}})+\! \sum_{i\in \hat{I}} \mu_{\bar b}(X^{\bar b}_{\hat{x}'_i,\hat U_{i}})
\end{equation*}
that
\begin{equation}\label{a3}
\sum_{i\in \hat{I}} \mu_{\bar b}(X^{\bar b}_{\hat{x}'_i,\hat U_{i}})\leq   \! \sum_{i\in {I}}| \mu_{\bar b}(X^{\bar b}_{{x}'_i, U_{i}})-   \mu_{\bar a}(X^{\bar a}_{{x}'_i, U_{i}}) |
\leq C| I| \psi \theta_0^{\frac{1}{\delta}}.
\end{equation}
Finally, the result follows from \eqref{a1}, \eqref{a2} and \eqref{a3}.
 \end{proof}


\begin{thebibliography}{10}





\bibitem{Du} F. Durand, Linearly recurrent subshifts have a finite number of non-periodic
subshift factors. Ergodic Theory Dynam. Systems, 20, \textbf{4}  (2000) 1061--1078.

\bibitem{Fe} S. Ferenczi, Rank and symbolic complexity. Ergodic Theory Dynam. Systems,
16, \textbf{4} (1996) 663--682.

\bibitem{F} N. P. Frank , A primer on substitution tilings of the Euclidean plane. Expositiones Mathematicae,
volume 26, \textbf{4} (2008)  295--326.

\bibitem{FS} N. P. Frank and L. Sadun, Fusion: a general framework for hierarchical tilings of $\mathbb{R}^d$. http://arxiv.org/abs/1101.4930.



\bibitem{GM} F. G\"{a}hler and G. Maloney, Cohomology of one-dimensional mixed substitution tiling spaces. arXiv:1112.1475 (2011).

\bibitem{GH} C.P.M. Geerse and A. Hof, Lattice gas models on self-similar aperiodic tilings. Rev. Math. Phys. \textbf{{3}} (1991) 163--221.



\bibitem{Go} W.H. Gottschalk, Orbit-closure decomposition and almost periodic properties. Bull. Amer. Math. Soc., \textbf{{50}} (1944) 915--919.


\bibitem{GS} B. Gr\"{u}nbaum and G. C. Shephard, \emph{Tilings and Patterns}. Freeman, New York 1986.


\bibitem{LMS} J.-Y. Lee, R. V. Moody, and B. Solomyak, Pure Point Dynamical and Diffraction Spectra, Ann. Henri Poincaré \textbf{3} (2002), 1003--1018.




\bibitem{PV} R. Pacheco and H. Vilarinho, Metrics on Tiling Spaces, Local Isomorphism and an Application of Brown's Lemma. 	arXiv:1202.4902v1 (2012).


\bibitem{Rad} C. Radin and M. Wolff, Space Tilings and Local Isomorphism, Geometriae Dedicata \textbf{42} (1992), 355--360.


\bibitem{Rob} E. A. Robinson, Jr., Symbolic dynamics and tilings of $\mathbb{R}^d$.  \emph{Symbolic dynamics and its applications}, volume 60 of Proc. Sympos. Appl. Math.,   Amer. Math. Soc., Providence 2004, 81--119.
\bibitem{Rue} D. Ruelle, \emph{Statistical mechanics: Rigorous results}, W. A. Benjamin, Inc., New York - Amsterdam, 1969.


   \bibitem{Sol} B. Solomyak, Dynamics of Self-Similar Tilings, Ergodic Theory and Dynamical Systems \textbf{17} (1997), 695--738. Errata,  Ergodic Theory and Dynamical Systems \textbf{19} (1999), 1685.

\end{thebibliography}
\end{document}